\documentclass{article}
\usepackage[utf8]{inputenc}
\usepackage{amsmath}
\usepackage{amsthm}
\usepackage{geometry}
\usepackage{mathtools}
\usepackage{amssymb}
\usepackage{cite}

\usepackage{bbm}

\newtheorem{thm}{Theorem}
\newtheorem{lem}{Lemma}
\newtheorem{defi}{Definition}
\newtheorem{assu}{Assumption}
\newtheorem{rem}{Remark}
\newtheorem{cor}{Corollary}
\newtheorem{propo}{Proposition}

\newcommand\relentr[2]{\text{H}(#1\,\vert\,#2)\vert_{\mathcal{F}^n}}
\newcommand\relentrk[2]{\text{H}(#1\,\vert\,#2)\vert_{\mathcal{F}^k}}

\title{On the specific relative entropy between martingale diffusions on the line}

\author{Julio Backhoff-Veraguas\footnote{University of Vienna.  JB acknowledges support by the Austrian Science Fund (FWF) through projects Y 00782 and P 36835.} \and Clara Unterberger\footnote{University of Vienna.}}

\begin{document}

\maketitle
The specific relative entropy, introduced in the Wiener space setting by N.\ Gantert, allows to quantify the discrepancy between the laws of  potentially mutually singular measures. It appears naturally as the large deviations rate function in a randomized version of Donsker's invariance principle, as well as in a novel transport-information inequality recently derived by H.\ F\"ollmer. A conjecture, put forward by the aforementioned authors, concerns a closed form expression for the specific relative entropy between continuous martingale laws in terms of their quadratic variations. 
 We provide a first partial result in this direction, by establishing this conjecture in the case of well-behaved martingale diffusions on the line.

\section{Introduction}

Let $(\mathcal{C},\mathcal{F})$ be the Wiener space $\mathcal{C}=C([0,1],\mathbb{R})$ and $\mathcal{F}=\sigma(X_t:t\in[0,1])$ where $X_t\colon \mathcal{C}\to\mathbb{R}$ with $X_t(\omega)=\omega (t)$ is the canonical process. Define also $$\mathcal{F}^n \coloneqq \sigma (X_{\frac{k}{n}}:k=0,1,\dots,n).$$

The central object of study in this article is the concept of specific relative entropy, introduced in the setting of Wiener space by N.\ Gantert in her PhD thesis \cite{Ga91} and recently studied by the works of H.\ F\"ollmer \cite{Fo22,Fo22a}.

\begin{defi}
Given $\mathbb{Q},\mathbb{P}$ probability measures on $(\mathcal{C},\mathcal{F})$, 
the upper specific relative entropy of $\mathbb{Q}$ with respect to $\mathbb{P}$ is defined as
\begin{equation*}
    h^u\left( \mathbb{Q} \, | \,\mathbb{P} \right) 
    \coloneqq \limsup_{n\to\infty} \frac{1}{n} \relentr{\mathbb{Q}}{\mathbb{P}},
\end{equation*}
where $\relentr{\mathbb{Q}}{\mathbb{P}}$ is the relative entropy of $\mathbb Q$ with respect to $\mathbb P$ restricted to $\mathcal{F}^n$, namely
\begin{equation*}
    \relentr{\mathbb{Q}}{\mathbb{P}} 
    \coloneqq \begin{cases}
    \int_{\mathcal{C}} \log\frac{d\mathbb{Q}}{d\mathbb{P}}\vert_{\mathcal{F}^n} \, d\mathbb{Q} & \text{if $\,\mathbb{Q}\vert_{\mathcal{F}^n} \ll \mathbb{P}\vert_{\mathcal{F}^n} $}  \\
    +\infty  & \text{otherwise}.
  \end{cases}
\end{equation*}

Similarly the lower specific relative entropy of $\mathbb{Q}$ with respect to $\mathbb{P}$ is defined as
\begin{equation*}
    h^\ell\left( \mathbb{Q} \, | \,\mathbb{P} \right) 
    \coloneqq \liminf_{n\to\infty} \frac{1}{n} \relentr{\mathbb{Q}}{\mathbb{P}}.
\end{equation*}
Finally, if $h^\ell( \mathbb{Q} \, | \,\mathbb{P})=h^u( \mathbb{Q} \, | \,\mathbb{P})$, we denote this common value by $h( \mathbb{Q} \, | \,\mathbb{P})$ and call it the specific relative entropy of $\mathbb{Q}$ with respect to $\mathbb{P}$.
\end{defi}

\begin{rem} $ $

\begin{enumerate}
\item
The specific relative entropy is meaningful even in situations where measures singular to each other are being compared. This is the case of continuous martingale laws, which typically have infinite relative entropy but may still have a finite specific relative entropy when compared with each other. In \cite[Kapitel II.4]{Ga91} it was shown that $h(\cdot|\mathbb W)$, with $\mathbb W$ being Wiener measure, is the rate function in a large deviations principle associated to a randomized Donsker-type approximation of Brownian motion.

\item
We took the liberty to introduce the terminology of ``upper/lower specific relative entropies''. In fact in \cite{Fo22} already $h^\ell$ is called specific relative entropy, whereas \cite{Ga91} restricts itself to the case when $h^\ell$ and $h^u$ coincide. The works \cite{Fo22,Fo22a} considers rather the subsequence of dyadic sigma-algebras $\mathcal F^{2^n}$ in the definitions of specific relative entropy. 
\item
Both \cite{Fo22,Fo22a} and \cite{Ga91} mostly study the specific relative entropy for $\mathbb{Q},\mathbb{P}$ being general martingale laws, or even more general stochastic processes than martingales. In the present work we will specialize the discussion to $\mathbb{Q},\mathbb{P}$ being the laws of very well-behaved \emph{martingale diffusions}. 

\item The concept of specific relative entropy is older than the work \cite{Ga91}. For instance this concept appears in Donsker and Varadhan's study  of large deviations for stationary Gaussian processes \cite{DoVa85} and in Föllmer's study  of lattice models \cite{Fo88}.
 \end{enumerate}
\end{rem}

Let $B=(B_t)_{t\in[0,1]}$ be a one-dimensional Brownian motion on some filtered probability space $(\Omega,\mathcal{S},(\mathcal{S}_t)_{t\in[0,1]},\mathbb{S})$.
Let $M^x=(M^x_t)_{t\in[0,1]}$ be the solution of
\begin{equation*}
    \begin{cases}
    dM^x_t = \sigma (M^x_t) \, dB_t \\
    M^x_0=x,
    \end{cases}
\end{equation*}
and likewise let $N^x=(N^x_t)_{t\in[0,1]}$ be the solution of
\begin{equation*}
    \begin{cases}
    dN^x_t = \eta (N^x_t) \, dB_t \\
    N^x_0=x.
    \end{cases}
\end{equation*}
Throughout this article we assume rather strong regularity assumptions on the diffusion coefficients $\sigma,\eta$. They will guarantee a unique strong solution to the above stochastic differential equations in particular,  but more importantly, they will allow us to gauge the small-time behaviour of $M^x$ and $N^x$. The precise assumptions are:

\begin{assu}\label{assu_regularity}
The coefficients $\sigma,\eta\colon \mathbb{R}\to\mathbb{R}_+$ are twice continuously differentiable and such that $$ \sigma(\cdot),\eta(\cdot) \in (\delta, {1}/{\delta}),$$ for some $0<\delta\leq 1$. Moreover, for some $L\in\mathbb R$ we have $$\max\{\lVert\sigma'\rVert_{\infty},\lVert\sigma''\rVert_{\infty},\lVert\eta'\rVert_{\infty},\lVert\eta''\rVert_{\infty}\}<L.$$
\end{assu} 

The main result of this article establishes the existence of the specific relative entropy between the laws of $M^x$ and $N^x$, and provides a closed form expression for this quantity:

\begin{thm} \label{formula}
Under Assumption \ref{assu_regularity}, let $M^x$ and $N^x$ be as above and call $\mathbb{Q}^x$ and ${\mathbb{P}^x}$  their respective  laws in $(\mathcal{C},\mathcal{F})$.
Then 
\begin{itemize}
\item[(i)] the specific relative entropy of $\mathbb{Q}^x$ with respect to $\mathbb{P}^x$ exists, and
\item[(ii)] 
 it has the closed form 
\begin{equation}\label{equformula}
    h(\mathbb{Q}^x\,|\,{\mathbb{P}^x}) 
    = \frac{1}{2} \mathbb{E}\left[ \int_0^1 \left\{\frac{\sigma(M^x_s)^2}{\eta(M^x_s)^2} -1 - \log\frac{\sigma(M^x_s)^2}{\eta(M^x_s)^2} \right\}\,ds \right].
\end{equation}
\end{itemize}

\end{thm}

Recalling that $X$ stood for the canonical process, Formula \eqref{equformula} becomes
$$h(\mathbb{Q}^x\,|\,{\mathbb{P}^x})  = \frac{1}{2}\mathbb E_{\mathbb Q^x}\left[ \int_0^1 \left\{ \frac{\sigma(X_s)^2}{\eta(X_s)^2} -1 - \log\frac{\sigma(X_s)^2}{\eta(X_s)^2}  \right\} ds \right].$$
%
%
A formula of this very kind had been conjectured in \cite[Kapitel I.1]{Ga91} in the case that $\eta\equiv 1$. Specifically, calling $\mathbb{W}$ the Wiener measure and $\mathbb Q$ a martingale law (w.l.o.g.\ both started at zero), the conjecture asserted the equality 
$$h(\mathbb{Q}\,|\,{\mathbb{W}})  = \frac{1}{2}\mathbb E_{\mathbb Q}\left[ \int_0^1 \left\{ \sigma(X_s)^2 -1 - \log\sigma(X_s)^2  \right\} ds \right].$$
In fact, it was already established in \cite{Ga91,Fo22,Fo22a} that the r.h.s.\ here is always the smaller one, and this fact has already found an application in \cite{Fo22,Fo22a} were a novel transport-information inequality was derived. Yet to the best of our knowledge the conjecture (i.e.\ the equality of left- and right-hand sides) had only been confirmed in the following cases: for Gaussian martingales, i.e.\ when $\sigma$ is deterministic (that is to say, not a function of space), for the case of geometric Brownian motion, and for the case of $M_t=\int_0^tB_sdB_s$. Our contribution is hence to provide a first corroboration of this conjecture beyond the Gaussian case in a rather systematic way. Furthermore we also allowed for $\eta$ being non-constant. This seems relevant to us, as one could be possibly interested in other reference martingales than Brownian motion. The proof of Formula \eqref{equformula} is achieved by careful expansion of the transition densities of the involved martingales. This goes some way into explaining why we restrict ourselves to martingale diffusions with fairly regular coefficients. 


As a corollary to Theorem \ref{formula} we are able to obtain a statistical estimator for the specific relative entropy based on independent samples. As in functional data analysis \cite{WaChMu16} (in the so-called asymptotic dense regime), the idea is that we have access to iid sample paths of a continuous process which we can observe only at mesh points, the size of the mesh going to zero as the number of sampled paths grows (in particular, no sampled path can be observed at more than a finite number of time instances). This is the content of our second main result: 

\begin{cor}\label{corollary}
Under Assumption \ref{assu_regularity}, let $M^x$ and $N^x$ be as before and call $\mathbb{Q}^x$ and ${\mathbb{P}^x}$  their respective  laws in $(\mathcal{C},\mathcal{F})$. Let $\{M^{x,n}\}_{n\in\mathbb N}$ be iid samples from  $\mathbb{Q}^x$. Then we have the almost sure convergence
\begin{align}
\lim_{n\to\infty}\frac{1}{n}\sum_{k\leq n}\frac{1}{k} \log\left( \frac{ d \mathbb Q^x}{d\mathbb P^x}\Big \vert_{\mathcal F^k} \Big . \big (M^{x,k}_{\frac{0}{k}},M^{x,k}_{\frac{1}{k}},\dots, M^{x,k}_{\frac{k}{k}}\big )  \right) =  h(\mathbb{Q}^x\,|\,{\mathbb{P}^x}).
\end{align} 
\end{cor}

The proof of this corollary is based on Theorem \ref{formula}, a strong law of large numbers for independent (possibly non-identically distributed) samples, and the short time behaviour of the transition densities of the involved martingales.

In our last main result we obtain an almost-sure counterpart to Theorem \ref{formula} along dyadic partitions. For simplicity we restrict ourselves to the case when the reference martingale $N^x$ is Brownian motion started at $x$.

\begin{propo}\label{prop:as}
Suppose that $\eta\equiv 1$ and that $\sigma$ fulfills the conditions in Assumption \ref{assu_regularity}. Then we have the almost-sure limit
\begin{align}
\lim_{k\to\infty}\frac{1}{2^k}\log\left( \frac{ d \mathbb Q^x}{d\mathbb P^x}\Big \vert_{\mathcal F^{2^k}} \Big . \big (M^{x}_{\ell\cdot 2^{-k}}\big )_{\ell=0}^{2^k}  \right) =
\frac{1}{2}\int_0^1 \left\{\sigma(M^x_s)^2 -1 - \log\sigma(M^x_s)^2 \right\}\,ds 
\end{align}
\end{propo}

This proposition establishes a conjecture in \cite[p.\ 15]{Ga91} in the case when $M$ is a martingale diffusion satisying our strong regularity assumptions. Incidentally, our proof reveals that the same result is true if we replace the role of $2^k$ by $n_k$, as long as $\sum_{k=0}^\infty n_k<\infty$.

\section{Proof of Theorem \ref{formula}}

The proof of Theorem $\ref{formula}$ relies on the following rather precise estimates on the transition density function of martingale diffusions as above.  Throughout we denote
$$d_{\sigma}(a,b):=\left | \int_a^b \frac{du}{\sigma(u)}\right |.$$
The metric $d_{\sigma}(a,b)$ is the geodesic distance induced by $\sigma$ (more precisely, by $\sigma^2$) on $\mathbb R$.

\begin{lem} \label{density}
Let $\sigma$ and $M^x$ be as above and let $p(t,x,y)$ be the transition density function of $M^x$ from $x$ at time $0$ to $y$ at time $t$.
Under Assumption \ref{assu_regularity} there are  constants $C_1$ and $C_2$, depending only on $\delta$ and $L$, such that, for all $t\in(0,1]$ and $x,y\in\mathbb{R}$:
\begin{equation*}
   e^{-C_2 t} \frac{1}{\sqrt{2\pi t}} \sqrt{\frac{\sigma (x)}{\sigma (y)}}\frac{1}{\sigma (y)} e^{\left( - d_\sigma(x,y)^2/2t \right)}\leq p(t,x,y) 
    \leq e^{C_1 t} \frac{1}{\sqrt{2\pi t}} \sqrt{\frac{\sigma (x)}{\sigma (y)}} \frac{1}{\sigma (y)} e^{\left( -d_\sigma(x,y)^2/2t \right) } .
\end{equation*}

\end{lem}

The content of Lemma \ref{density} is surely known to experts, as the study of asymptotic expansions for the heat kernel is a classical subject. 
 On the other hand Lemma \ref{density} is  non-asymptotic and of global nature, which we find useful. It seems to us that this result should follow from the precise estimates in \cite[Corollary 3.1 and Theorem 4.1]{LiYa86}, but we cannot guarantee this assertion as we lack geometric training. Hence we rather provide our own self-contained proof:

\begin{proof}[Proof of Lemma \ref{density}]
We use throughout that $B$ is a Brownian motion started at the origin. Let $g\colon\mathbb{R}\to\mathbb{R}$ be the strictly increasing and surjective function defined by $$g(y) \coloneqq \int_x^y \frac{1}{\sigma(u)} \,du,$$ and $g^{-1}$ its inverse. By Itô's lemma we have $$g(M_t^x)=B_t-\frac{1}{2}\int_0^t \sigma'(M_s^x) \,ds.$$ Hence $Y_t\coloneqq g(M_t^x)$ satisfies $Y_t=B_t+\int_0^t b(Y_s)ds$, where $b\coloneqq -\frac{1}{2}\sigma' \circ g^{-1}$.

For any $\Phi \geq 0$ measurable and bounded, we have
\begin{align}
    \mathbb{E}\left[ \Phi (M_t^x) \right] 
    &= \mathbb{E}\left[ \Phi(g^{-1}(Y_t)) \right] \nonumber\\
    &= \mathbb{E}\left[ \Phi(g^{-1}(B_t))\exp{\left( \int_0^t b(B_s) \,dB_s -\frac{1}{2} \int_0^t b(B_s)^2  \,ds\right)} \right] \nonumber\\
    &= \mathbb{E}\left[ \Phi(g^{-1}(B_t))\exp{\left( \int_0^{B_t} b(u) \,du -\frac{1}{2} \int_0^t b'(B_s)  \,ds - \frac{1}{2} \int_0^t b(B_s)^2 \,ds \right)} \right] \label{ErwartungswertPhi}
\end{align}
where the second equality is due to Girsanov's theorem (applicable as $\sigma'$ is bounded), whereas the last equality follows from 
\begin{equation*}
    \int_0^{B_t} b(u) \,du 
    = \int_0^t b(B_s) \,dB_s + \frac{1}{2} \int_0^t b'(B_s) \,ds,
\end{equation*}  
which is just Itô's lemma applied to the function $f(z)\coloneqq \int_0^z b(u) \,du$. 

To get an upper bound for $\eqref{ErwartungswertPhi}$ note that,
\begin{align} \label{dasmuessenwirersetzen} 
    -b'-b^2 
    \leq -b' 
    = \frac{1}{2}(\sigma \cdot \sigma '' ) \circ g^{-1}
    \leq \frac{L}{2\delta}
\end{align}
which implies,
\begin{equation*}
    \exp{\left( -\frac{1}{2} \int_0^t b'(B_s) \,ds - \frac{1}{2} \int_0^t b(B_s)^2 \,ds \right)} 
    \leq \exp{(t\frac{L}{4\delta})}.
\end{equation*}
Inserting in $\eqref{ErwartungswertPhi}$ with $C_1\coloneqq\frac{L}{4\delta}$ yields,
\begin{equation}\label{obereSchrankefuerErwartungswert}
    \mathbb{E}\left[ \Phi(M_t^x) \right]
    \leq \exp{(tC_1)} \mathbb{E}\left[ \Phi(g^{-1}(B_t))\exp{\left( \int_0^{B_t} b(u) \,du  \right)} \right] .
\end{equation}
Note that
\begin{equation*}
    b
    = -\frac{1}{2}\sigma' \circ g^{-1}
    = -\frac{1}{2} \frac{(\sigma\circ g^{-1})'}{\sigma \circ g^{-1}}
    = -\frac{1}{2} ( \log\circ\sigma\circ g^{-1})' 
\end{equation*}
and $g^{-1}(B_0)=g^{-1}(0)=x$. Therefore
\begin{equation*}
    \int_0^{B_t} b(u) \,du 
    = -\frac{1}{2} \log\left( \frac{\sigma(g^{-1}(B_t))}{\sigma(g^{-1}(B_0))} \right)
    = \log\sqrt{  \frac{\sigma(x)}{\sigma(g^{-1}(B_t))} }
\end{equation*}
We can now rewrite the expectation on the right-hand side of \eqref{obereSchrankefuerErwartungswert} by using this observation and then expressing it in terms of the density of $B_t$. The desired representation then follows from a change of variables with $y=g^{-1}(z)$, namely:
\begin{align*}
   & \mathbb{E}\left[ \Phi(g^{-1}(B_t))\exp{\left( \int_0^{B_t} b(u) \,du \right)} \right] \\
    =& \mathbb{E}\left[ \Phi(g^{-1}(B_t)) \sqrt{ \frac{\sigma(x)}{\sigma(g^{-1}(B_t))} } \right] \\
    =& \int_{\mathbb{R}} \Phi(g^{-1}(z)) \sqrt{ \frac{\sigma(x)}{\sigma(g^{-1}( z ))} }     \frac{1}{\sqrt{2\pi t}}\exp{\left( -\frac{z^2}{2t}\right)} \,dz \\
    =& \int_{\mathbb{R}} \Phi(y)\sqrt{ \frac{\sigma(x)}{\sigma(y)} } \frac{1}{\sqrt{2\pi t}}\exp{\left( -\frac{g(y)^2}{2t}\right)} g'(y) \,dy \\
    =& \int_{\mathbb{R}} \Phi(y)\sqrt{ \frac{\sigma(x)}{\sigma(y)} }  \frac{1}{\sqrt{2\pi t}}\exp{\left( - \frac{1}{2t} \left(\int_x^y\frac{1}{\sigma(u)} \,du\right)^2 \right)} \frac{1}{\sigma(y)} \,dy.
\end{align*}
We conclude that
\begin{equation*}
   \mathbb{E}\left[ \Phi(M_t^x) \right]
    \leq \exp{(tC_1)}\int_{\mathbb{R}} \Phi(y) \sqrt{\frac{\sigma(x)}{\sigma(y)}} \frac{1}{\sqrt{2\pi t}}\exp{\left( -\frac{1}{2t} d_\sigma(x,y)^2 \right)} \frac{1}{\sigma(y)} \,dy .
\end{equation*}
As this holds for all $\Phi\geq 0$ measurable and bounded, this proves the desired upper bound.

The lower bound follows analogously by replacing \eqref{dasmuessenwirersetzen} with
\begin{equation*} 
    -b'-b^2  
    = \frac{1}{2}(\sigma \sigma '')\circ g^{-1} - (\frac{1}{2}\sigma '\circ g^{-1})^2 
    \geq -\left( \frac{L\delta}{2}+\frac{L^2}{4} \right) .
\end{equation*}
This implies
\begin{equation*}
    \exp{\left( -\frac{1}{2} \int_0^t b'(B_s) \,ds - \frac{1}{2} \int_0^t b(B_s)^2 \,ds \right)} 
    \geq \exp{( -tC_2 )},
\end{equation*}
where $C_2\coloneqq \frac{L\delta}{4}+\frac{L^2}{8}$.
\end{proof}

We are now in a position to prove Theorem $\ref{formula}$.

\begin{proof}[Proof of Theorem \ref{formula}]
Let $q(t,x,y)$ and $p(t,x,y)$ be the transition density functions of $M^x$ and $N^x$ respectively. To ease notation we now drop the superscript $x$ from $M$ and $N$.
Since $M$ and $N$ are Markov and time-homogeneous, we have by the tensorization property of the relative entropy that
\begin{align*}
    \relentr{\mathbb{Q}^x}{{\mathbb{P}^x}}
    &= \text{H}\left( \mathcal{L}(M_0,M_{\frac{1}{n}}, \dots , M_1) \mid \mathcal{L}(N_0,N_{\frac{1}{n}}, \dots , N_1)  \right) \\
    &= \sum_{k=1}^n \int_{\mathbb{R}^2} \log\frac{ d\mathcal{L}( M_{\frac{k}{n}}\vert M_{\frac{k-1}{n}} =x_{k-1} )(x_k)  }{d\mathcal{L}( N_{\frac{k}{n}}\vert N_{\frac{k-1}{n}} =x_{k-1}  )(x_k) } \, d\mathcal{L}(M_{\frac{k-1}{n}},M_{\frac{k}{n}})(x_{k-1},x_k) \\
    &= \sum_{k=1}^n \int_{\mathbb{R}^2} \log \frac{ q(\frac{1}{n},x_{k-1},x_k) }{ p(\frac{1}{n},  x_{k-1},x_k)  }  \,d\mathcal{L}(M_{\frac{k-1}{n}},M_{\frac{k}{n}})(x_{k-1},x_k) \\
    &= \sum_{k=1}^n \mathbb{E}\left[   \log \frac{ q(\frac{1}{n},M_{\frac{k-1}{n}},M_{\frac{k}{n}}) }{ p(\frac{1}{n}, M_{\frac{k-1}{n}},M_{\frac{k}{n}})  }  \right]
\end{align*}
By Lemma $\ref{density}$ for the upper bound of $q$ and the lower bound of $p$, we derive the existence of a constant $C$ such that for all $x,y\in\mathbb{R},t\in(0,1]$:
\begin{equation*}
     \log\frac{q(t,x,y)}{p(t,x,y)} 
     \leq Ct + \frac{1}{2}\log\frac{\sigma(x)\eta(y)}{\sigma(y)\eta(x)} - \log\frac{\sigma(y)}{\eta(y)} -\frac{1}{2t}d_\sigma(x,y)^2 + \frac{1}{2t}d_\eta(x,y)^2.
\end{equation*}
Therefore, for all $n\in\mathbb{N}$ and any $k=1,\dots,n$:
\begin{align}
      &\mathbb{E}\left[\log \frac{ q(\frac{1}{n},M_{\frac{k-1}{n}},M_{\frac{k}{n}}) }{ p(\frac{1}{n}, M_{\frac{k-1}{n}},M_{\frac{k}{n}})  }  \right]      \nonumber\\
     &\leq \mathbb{E}\Bigg[ \frac{C}{n} + \frac{1}{2}\log  \frac{\sigma(M_{\frac{k-1}{n}})\eta( M_{\frac{k}{n}}  )}{\sigma(M_{\frac{k}{n}})\eta(M_{\frac{k-1}{n}})} - \log\frac{\sigma (M_{\frac{k}{n}} )}{\eta (M_{\frac{k}{n}} )} -       \frac{n}{2} d_\sigma(M_{\frac{k-1}{n}},M_{\frac{k}{n}})^2+\frac{n}{2} d_\eta(M_{\frac{k-1}{n}},M_{\frac{k}{n}})^2   \Bigg]       . \label{inequ}  
\end{align} 
Summing over $k=1,\dots,n$ the $\log  \frac{\sigma(M_{\frac{k-1}{n}})\eta( M_{\frac{k}{n}}  )}{\sigma(M_{\frac{k}{n}})\eta(M_{\frac{k-1}{n}})}$ terms form a telescopic sum. So for all $n\in\mathbb{N}$:
\begin{align}
    &\frac{1}{n}  \relentr{\mathbb{Q}^x}{{\mathbb{P}^x}} \nonumber\\
    &\leq \frac{1}{n}\sum_{k=1}^n \mathbb{E}\Bigg[ \frac{C}{n} + \frac{1}{2}\log  \frac{\sigma(M_{\frac{k-1}{n}})\eta( M_{\frac{k}{n}}  )}{\sigma(M_{\frac{k}{n}})\eta(M_{\frac{k-1}{n}})}  - \log\frac{\sigma (M_{\frac{k}{n}} )}{\eta (M_{\frac{k}{n}} )}  -         \frac{n}{2} d_\sigma(M_{\frac{k-1}{n}},M_{\frac{k}{n}})^2+\frac{n}{2} d_\eta(M_{\frac{k-1}{n}},M_{\frac{k}{n}})^2   \Bigg]       \nonumber\\
    &= \frac{C}{n} + \mathbb{E}\Bigg[  \frac{1}{2n}\log \frac{\sigma(M_0)\eta(M_1)}{\sigma(M_1)\eta(M_0)}   - \frac{1}{n}\sum_{k=1}^n \log \frac{\sigma (M_{\frac{k}{n}}) }{\eta (M_{\frac{k}{n}}) }     
    -\frac{1}{2}\sum_{k=1}^n  d_\sigma(M_{\frac{k-1}{n}},M_{\frac{k}{n}})^2 +\frac{1}{2}\sum_{k=1}^n   d_\eta(M_{\frac{k-1}{n}},M_{\frac{k}{n}})^2     \label{obereSchrankeRelEntr}
    \Bigg]
\end{align}

We now show that $\eqref{obereSchrankeRelEntr}$ converges to the right-hand side of $\eqref{equformula}$. 
Firstly, since $\sigma$ and $\eta$ are bounded away from $0$ and bounded from above, it is clear that,
\begin{equation*}
     \lim_{n \to \infty}  \mathbb{E}\left[  \frac{1}{2n}\log \frac{\sigma(M_0)\eta(M_1)}{\sigma(M_1)\eta(M_0)}  \right] 
     = 0
\end{equation*}
Furthermore for almost every $\omega\in\Omega$ the map $t\mapsto \log(\frac{\sigma(M_t(\omega))}{\eta(M_t(\omega))})$ is continuous, so that
\begin{equation*}
    \sum_{k=1}^n \log\frac{\sigma(M_{\frac{k}{n}}(\omega))}{\eta(M_{\frac{k}{n}}(\omega))}\mathbbm{1} _{ \left(\frac{k-1}{n},\frac{k}{n} \right] }(t) 
    \xrightarrow{n \to \infty}{} \log\frac{\sigma(M_t(\omega))}{\eta(M_t(\omega))} \qquad \text{Leb}\otimes\mathbb{P}-a.e.
\end{equation*}
Since the left-hand side above is bounded uniformly in $n$, we have by dominated convergence
\begin{align*}
    \lim_{n \to \infty} \mathbb{E}\left[  \frac{1}{n}\sum_{k=1}^n \log\frac{\sigma(M_{\frac{k}{n}})}{\eta(M_{\frac{k}{n}})}  \right] 
     &=\lim_{n \to \infty}\mathbb{E}\left[  \int_0^1 \sum_{k=1}^n  \log\frac{\sigma(M_{\frac{k}{n}})}{\eta(M_{\frac{k}{n}})}  \mathbbm{1} _{ \left(\frac{k-1}{n},\frac{k}{n} \right] }(s)  \,ds \right] \\
    &= \mathbb{E}\left[  \int_0^1 \log\frac{\sigma (M_s)}{\eta(M_s)}  \,ds \right] .
\end{align*}
Next, Itô's lemma applied to the function $F(z):=\int_0^z\frac{1}{\eta(u)}\,du$ yields,
\begin{equation*}
   F(M_t)= \int_0^{M_t} \frac{1}{\eta(u)} \,du 
    = \int_0^t \frac{\sigma(M_u)}{\eta(M_u)} \,dB_u - \frac{1}{2} \int_0^t \frac{\eta'(M_u)}{\eta(M_u)^2}\sigma(M_u)^2 \,du.
\end{equation*}
Thus we have
\begin{align*}
\mathbb{E}\left[\frac{1}{2}\sum_{k=1}^n d_\eta(M_{\frac{k-1}{n}},M_{\frac{k}{n}})^2   \right]
     &=\mathbb{E}\left[  \frac{1}{2}\sum_{k=1}^n  \left( \int_{M_{\frac{k-1}{n}}}^{M_{\frac{k}{n}}} \frac{1}{\eta(u) } \,du \right)^2 \right]\\
     &= \mathbb{E}\left[  \frac{1}{2}\sum_{k=1}^n   \left(  F(M_{\frac{k}{n}}) - F(M_{\frac{k-1}{n}})   \right)^2 \right].
\end{align*}
At this point one can directly deduce 
\begin{equation}
    \lim_{n\to\infty}\mathbb{E}\left[\frac{1}{2}\sum_{k=1}^n d_\eta(M_{\frac{k-1}{n}},M_{\frac{k}{n}})^2   \right]
    = \mathbb{E}\left[\int_0^1  \frac{1}{2} \left( \frac{\sigma(M_u)}{\eta(M_u)} \right)^2 \,du \right],\label{eq:Riemannian_sum}
\end{equation}
from the fact that $$\mathbb{E}\left[  \frac{1}{2}\sum_{k=1}^n   \left(  F(M_{\frac{k}{n}}) - F(M_{\frac{k-1}{n}})   \right)^2 \right]\to \mathbb E[\langle F(M) \rangle_1],$$ as follows directly per properties of the quadratic variation of Itô processes with bounded coefficients. To be self-contained we provide now a detailed argument:
\begin{align*}
     &\mathbb{E}\left[  \frac{1}{2}\sum_{k=1}^n   \left(  F(M_{\frac{k}{n}}) - F(M_{\frac{k-1}{n}})   \right)^2 \right] \\
     &= \frac{1}{2}\sum_{k=1}^n \mathbb{E}\Bigg[ \left( \int_{\frac{k-1}{n}}^{\frac{k}{n}} \frac{\sigma(M_u)}{\eta(M_u)} \,dB_u \right)^2 - \int_{\frac{k-1}{n}}^{\frac{k}{n}} \frac{\sigma(M_u)}{\eta(M_u)} \,dB_u \int_{\frac{k-1}{n}}^{\frac{k}{n}} \frac{\eta'(M_u)}{\eta(M_u)^2}\sigma(M_u)^2 \,du + \\
     &\hspace{9cm}
     \frac{1}{4} \left(\int_{\frac{k-1}{n}}^{\frac{k}{n}} \frac{\eta'(M_u)}{\eta(M_u)^2}\sigma(M_u)^2 \,du\right)^2  \Bigg] \\
     &= \mathbb{E}\Bigg[ \frac{1}{2}\int_0^1 \left(\frac{\sigma(M_u)}{\eta(M_u)}\right)^2 \,du - \frac{1}{2}\sum_{k=1}^n \int_{\frac{k-1}{n}}^{\frac{k}{n}} \frac{\sigma(M_u)}{\eta(M_u)} \,dB_u \int_{\frac{k-1}{n}}^{\frac{k}{n}} \frac{\eta'(M_u)}{\eta(M_u)^2}\sigma(M_u)^2 \,du + \\ &\hspace{9cm}
     \frac{1}{2}\sum_{k=1}^n\frac{1}{4}\left(\int_{\frac{k-1}{n}}^{\frac{k}{n}} \frac{\eta'(M_u)}{\eta(M_u)^2}\sigma(M_u)^2 \,du\right)^2  \Bigg] .
\end{align*}
The expression above converges to $\mathbb{E}\left[ \frac{1}{2}\int_0^1 \left(\frac{\sigma(M_u)}{\eta(M_u)}\right)^2 \,du\right]$ as $n\to\infty$. Indeed, since $\lVert\eta '\rVert_{\infty}<L$ and $\sigma,\eta\in(\delta,\frac{1}{\delta})$, we have that
\begin{align*}
    \mathbb{E}\left[\left \lvert \int_{\frac{k-1}{n}}^{\frac{k}{n}} \frac{\sigma(M_u)}{\eta(M_u)} \,dB_u \right \rvert \right] 
   & \leq \left( \mathbb{E}\left[ \left(  \int_{\frac{k-1}{n}}^{\frac{k}{n}} \frac{\sigma(M_u)}{\eta(M_u)} \,dB_u   \right)^2 \right]  \right)^{\frac{1}{2}} \\
   & = \left( \mathbb{E}\left[ \int_{\frac{k-1}{n}}^{\frac{k}{n}} \frac{\sigma(M_u)^2}{\eta(M_u)^2} \,du    \right]  \right)^{\frac{1}{2}} \\
   & \leq \sqrt{ \frac{1}{n}}\frac{1}{\delta^2},
\end{align*}
and therefore as $n\to\infty$
\begin{equation*}
    \left| \sum_{k=1}^n \int_{\frac{k-1}{n}}^{\frac{k}{n}} \frac{\sigma(M_u)}{\eta(M_u)} \,dB_u \int_{\frac{k-1}{n}}^{\frac{k}{n}} \frac{\eta'(M_u)}{\eta(M_u)^2}\sigma(M_u)^2 \,du   \right| 
    \leq \frac{L}{\delta^4}\frac{1}{n} \sum_{k=1}^n  \left|  \int_{\frac{k-1}{n}}^{\frac{k}{n}} \frac{\sigma(M_u)}{\eta(M_u)} \,dB_u   \right|  
    \xrightarrow{L^1} 0 .
\end{equation*}
Moreover,
\begin{equation*}
  \sum_{k=1}^n \left( \int_{\frac{k-1}{n}}^{\frac{k}{n}} \frac{\eta'(M_u)}{\eta(M_u)^2}\sigma(M_u)^2 \,du \right)^2  
  \leq  \sum_{k=1}^n \frac{L^2}{\delta^8}\frac{1}{n^2} 
  = \frac{L^2}{\delta^8}\frac{1}{n}
  \xrightarrow{L^1} 0 \quad \text{as $n\to\infty$}.
\end{equation*}
All in all, we have established  \eqref{eq:Riemannian_sum}. Applying the same arguments with $\eta$ replaced by $\sigma$, shows 
\begin{equation*}
    \lim_{n\to\infty}  \mathbb{E}\left[\frac{1}{2}\sum_{k=1}^n d_\sigma(M_{\frac{k-1}{n}},M_{\frac{k}{n}})^2   \right] 
    = \frac{1}{2}.
\end{equation*}

Finally we conclude that
\begin{equation*}
   h^u({\mathbb{Q}^x}|{\mathbb{P}^x}):= \limsup_{n\to\infty} \frac{1}{n}\relentr{\mathbb{Q}^x}{{\mathbb{P}^x}}  
    \leq \frac{1}{2}\mathbb{E}\left[ \int_0^1  \frac{\sigma(M_s)^2}{\eta(M_s)^2} - 1 - \log\frac{\sigma(M_s)^2}{\eta(M_s)^2} \,ds\right].
\end{equation*}
 Analogously, by the lower bound for $q$ and the upper bound for $p$ from Lemma $\ref{density}$, applied in $\eqref{inequ}$, we get:
 \begin{equation*}
  h^\ell({\mathbb{Q}^x}|{\mathbb{P}^x}):=  \liminf_{n\to\infty} \frac{1}{n} \relentr{\mathbb{Q}^x}{{\mathbb{P}^x}}  
    \geq \frac{1}{2} \mathbb{E}\left[\int_0^1 \frac{\sigma(M_s)^2}{\eta(M_s)^2} - 1 - \log\frac{\sigma(M_s)^2}{\eta(M_s)^2} \,ds\right],
\end{equation*}
which finishes the proof.
\end{proof}

\begin{rem}
We find the limit in \eqref{eq:Riemannian_sum}, equivalently
\begin{equation*}
    \lim_{n\to\infty}\mathbb{E}_{\mathbb Q^x}\left[\sum_{k=1}^n d_\eta(X_{\frac{k-1}{n}},X_{\frac{k}{n}})^2   \right]
    =  \mathbb{E}_{\mathbb Q^x}\left[\int_0^1  \eta(X_u)^{-2} d\langle X\rangle_u\right ] ,
\end{equation*}
curious and worth highlighting. We suspect that there is a geometric interpretation to this result, but we do not pursue this line of reasoning here. 
\end{rem}


\section{Proof of Corollary \ref{corollary}}

In this part it will be convenient to apply a weaker version of Lemma \ref{density}, namely: 

\begin{lem} \label{density_2}
Let $\sigma$, $M^x$ and $p(t,x,y)$ as in Lemma \ref{density}.
Under Assumption \ref{assu_regularity} there is a constant $C$, depending only on $\delta$ and $L$, such that, for all $t\in(0,1]$ and $x,y\in\mathbb{R}$:
\begin{equation*}
   e^{-C t} \frac{\delta^2}{\sqrt{2\pi t}} e^{\left( - (x-y)^2/2t\delta^2 \right)}\leq p(t,x,y) 
    \leq e^{C t} \frac{1}{\delta^2\sqrt{2\pi t}} e^{\left( - \delta^2(x-y)^2/2t \right)} .
\end{equation*}
\end{lem}

This follows directly from Lemma \ref{density} and Assumption \ref{assu_regularity} on $\sigma$, so we omit the proof. We remark that at least the upper bound in Lemma \ref{density_2} can be obtained in far more generality with PDE methods (see Theorems 4.5 and 5.4 in \cite[Chapter 6]{Fr10}).

We put ourselves now in the setting of Corollary \ref{corollary}, and introduce some notation here: If $\mathbb{Q}$ is a probability measure on $\mathcal{C}$ and if $\mathbb{Q} \vert_{\mathcal F^n}$ possesses a density with respect to $n$-dimensional Lebesgue measure\footnote{To be fully precise: the first marginal of $\mathbb{Q} \vert_{\mathcal F^n}$ is a Dirac measure. But as all martingales we consider are started at a fixed $x$ we may ``drop'' the first marginal and consider the density of the remaining arguments w.r.t.\ Lebesgue.}, we denote this density by the lower-case symbol $q_n$. If $X\in \mathcal C$ we denote $X_{(n)}:=(X_{i/n})_{i=0}^n$. Written in this terms, our aim is to establish the almost sure convergence
\begin{align}\label{eq:as_conv_2}
\lim_{n\to\infty}\frac{1}{n}\sum_{k\leq n}\frac{1}{k} \log\left( \frac{ q^x_k}{p^x_k}(M^{x,k}_{(k)})  \right) =  h(\mathbb{Q}^x\,|\,{\mathbb{P}^x}).
\end{align} 

We now omit the $x$-dependence everywhere, in order to keep notation at bay.

\begin{proof}[Proof of Corollary \ref{corollary}]
We will use the following Kolmogorov's version of the strong law of large numbers (KSLLN) for independent non-identically distributed sequences: If $(Y_k)_k$ are independent, with finite second moments, and such that $\sum_k\frac{1}{k^2}\text{Var}(Y_k)<\infty$, then $\frac{1}{n}\sum_{k\leq n} (Y_k-\mathbb E[Y_k]) \to 0$ almost surely as $n\to\infty$. This can be found in \cite[Theorem 2.3.10]{SeSi94}. For us 
$$Y_k:=\frac{1}{k} \log\left( \frac{ q_k}{p_k}(M^{k}_{(k)})  \right ),$$
so by definition the $(Y_k)_k$ are independent and $\mathbb E[Y_k]=  \frac{1}{k} \relentrk{\mathbb{Q}}{\mathbb{P}}$. Per Theorem \ref{formula} we know $\frac{1}{k} \relentrk{\mathbb{Q}}{\mathbb{P}} \to h(\mathbb{Q}\,|\,{\mathbb{P}})$, hence $\frac{1}{n}\sum_{k\leq n} \mathbb E[Y_k] \to h(\mathbb{Q}\,|\,{\mathbb{P}})$. Thus if we can apply KSLLN then we will have derived \eqref{eq:as_conv_2} as desired. Thus it remains to verify that $\sum_k\frac{1}{k^2}\text{Var}(Y_k)<\infty$, but as $\sum_k\frac{1}{k^2} \mathbb E[Y_k]^2<\infty$ in our case, it is enough to verify 
\begin{align}
\label{eq:verify}
 \sum_k\frac{1}{k^2} \mathbb E[Y_k^2] <\infty.
\end{align}
To this end remark that $Y_k^2\leq \frac{2}{k^2}\left\{\log(q_k(M^{k}_{(k)}) )^2 + \log(p_k(M^{k}_{(k)}) )^2  \right \}$. We now establish 
\begin{align}
\label{eq:verify2}
 \sum_k\frac{1}{k^4} \mathbb E[  \log(\ell_k(M_{(k)}) )^2 ] <\infty,
\end{align}
where we let the symbol $\ell$ stand for either $p$ or $q$,  implying then \eqref{eq:verify}. Applying the Markov property we have $$\ell_k(M_{(k)})=\prod_{i=1}^k \, \ell(1/k, M_{(i-1)/k},M_{i/k}),$$  so taking logarithms, bounding the squared sum by a sum of squares, and applying Lemma  \ref{density_2}, we obtain
\begin{align*}\log(\ell_k(M_{(k)}) )^2 &\leq 4k \sum_{i\leq k}\left\{ \frac{C^2}{k^2} +\log(\delta^2)^2+\frac{1}{4}\log(k/2\pi)^2+\frac{k^2}{4\delta^4}(M_{(i-1)/k}-M_{i/k})^4 \right\} \\
&= 4C^2 + 4k^2\log(\delta^2)^2+k^2\log(k/2\pi)^2+\frac{k^3}{\delta^4}\sum_{i\leq k}(M_{(i-1)/k}-M_{i/k})^4 .
\end{align*}
To check \eqref{eq:verify2}, first observe $\sum_k (4C^2 + 4k^2\log(\delta^2)^2+k^2\log(k/2\pi)^2)/k^4<\infty$, as $\log(x)\leq x^r/r$ for all $x\geq 1$ and $r\in (0,1)$. Second, by the BDG inequality, we have 
$$
\mathbb E[(M_{(i-1)/k}-M_{i/k})^4]\leq c\mathbb E\left[\left(\int_{(i-1)/k}^{i/k} \sigma(M_t)^2 \,dt \right)^2 \right]\leq \frac{c}{\delta^4k^2},
$$
so $\sum_{i\leq k}\mathbb E[(M_{(i-1)/k}-M_{i/k})^4]\leq \frac{c}{\delta^4k} $
and  $\sum_k \frac{1}{k^4}\sum_{i\leq k} \frac{k^3}{\delta^4}\mathbb E[(M_{(i-1)/k}-M_{i/k})^4] \leq c' \sum_k \frac{1}{k^2}  <\infty$. 
\end{proof}

\section{Proof of Proposition \ref{prop:as}}

\begin{proof}[Proof of Proposition \ref{prop:as}] To ease notation we now drop $x$ as a superscript . Let $q(t,x,y)$ and $p(t,x,y)$ denote the transition density functions of respectively $M$ and $N=B$. As in the proof of Theorem \ref{formula} we start with the inequality
\begin{align}
      &\log \frac{ q(\frac{1}{n},M_{\frac{\ell-1}{n}},M_{\frac{\ell}{n}}) }{ p(\frac{1}{n}, M_{\frac{\ell-1}{n}},M_{\frac{\ell}{n}})  }       \nonumber\\
     &\leq \frac{C}{n} + \frac{1}{2}\log  \frac{\sigma(M_{\frac{\ell-1}{n}})}{\sigma(M_{\frac{\ell}{n}})} - \log\sigma (M_{\frac{\ell}{n}} ) -       \frac{n}{2} d_\sigma(M_{\frac{\ell-1}{n}},M_{\frac{\ell}{n}})^2+\frac{n}{2} |M_{\frac{\ell-1}{n}}-M_{\frac{\ell}{n}}|^2        . \label{inequ_as}  
\end{align} 
By the time-homogeneous Markov property we have
 \[\log\left( \frac{ d \mathbb Q}{d\mathbb P}\Big \vert_{\mathcal F^{2^k}} \Big . \big (M_{\ell\cdot 2^{-k}}\big )_{\ell=0}^{2^k}  \right)= \sum_{\ell=1}^{2^k}
\log \frac{ q(2^{-k},M_{(\ell-1)\cdot 2^{-k}},M_{\ell\cdot 2^{-k}}) }{ p(2^{-k}, M_{(\ell-1)\cdot 2^{-k}},M_{\ell\cdot 2^{-k}})  }  ,  
\]
so our pre-limit quantity of interest,
\[\frac{1}{2^k}\log\left( \frac{ d \mathbb Q}{d\mathbb P}\Big \vert_{\mathcal F^{2^k}} \Big . \big (M_{\ell\cdot 2^{-k}}\big )_{\ell=0}^{2^k}  \right),\]
is bounded from above by 
\[C2^{-k} + 2^{-k-1}\log  \frac{\sigma(M_1)}{\sigma(M_0)} - \alpha_k  - \beta_k      +\delta_k,\]
where
\begin{align*}
\alpha_k& := 2^{-k}\sum_{\ell=1}^{2^k} \log\sigma (M_{\ell\cdot 2^{-k}} ),\\
\beta_k&:=  \frac{1}{2} \sum_{\ell=1}^{2^k}d_\sigma(M_{(\ell-1)\cdot 2^{-k}},M_{\ell\cdot 2^{-k}})^2,\\
\delta_k&:= \frac{1}{2}\sum_{\ell=1}^{2^k} |M_{(\ell-1)\cdot 2^{-k}}-M_{\ell\cdot 2^{-k}}|^2.
\end{align*}
We have the almost-sure limit $\lim_k \alpha_k = \int_0^1\log\sigma(M_t)dt$, by convergence of Riemann sums, the function $t\mapsto\log\sigma(M_t)$ being continuous and bounded. We observe also that
\[d_\sigma(M_{(\ell-1)\cdot 2^{-k}},M_{\ell\cdot 2^{-k}})^2 = \left|B_{(\ell-1)\cdot 2^{-k}} - B_{\ell\cdot 2^{-k}}-\frac{1}{2}\int_{(\ell-1)\cdot 2^{-k}}^{\ell\cdot 2^{-k}} \sigma'(M_t)dt \right|^2,\]
by applying It\^o formula to the process $\int_0^{M_t} \frac{dy}{\sigma(y)}$. Hence
\begin{align*}2\beta_k=\sum_{\ell=1}^{2^k}&\left\{ |B_{(\ell-1)\cdot 2^{-k}} - B_{\ell\cdot 2^{-k}}|^2+ \frac{1}{4}\left|\int_{(\ell-1)\cdot 2^{-k}}^{\ell\cdot 2^{-k}} \sigma'(M_t)dt \right|^2-(B_{(\ell-1)\cdot 2^{-k}} \right .  \\ & \left . - B_{\ell\cdot 2^{-k}})\int_{(\ell-1)\cdot 2^{-k}}^{\ell\cdot 2^{-k}} \sigma'(M_t)dt\right\}.
\end{align*}
Now we observe the almost-sure limits $\sum_\ell |B_{(\ell-1)\cdot 2^{-k}} - B_{\ell\cdot 2^{-k}}|^2\to 1$ and, as $\sigma'$ is bounded, also $\sum_\ell |\int_{(\ell-1)\cdot 2^{-k}}^{\ell\cdot 2^{-k}} \sigma'(M_t)dt |^2=0$. This and Cauchy-Schwarz establish that $\beta_k\to 1/2$ almost-surely. Finally we remark that $2\delta_k\to\int_0^1\sigma(M_t)^2dt$ almost surely. \footnote{In case that the reader is not aware of this fact we next provide a self-contained argument: First we estimate for $r<t$ that $\mathbb E[|(M_t-M_r)^2-\int_r^t\sigma(M_s)^2ds|^2] = \mathbb E[|\int_r^t(M_s-M_r)dM_s|^2]\leq c|t-r|^2$, with $c$ only depending on $\|\sigma\|_{\infty}$. Then we observe that $\mathbb E[\{(M_t-M_r)^2-\int_r^t\sigma(M_s)^2ds\}\,|\, \mathcal F_r]=\mathbb E[\int_r^t(M_s-M_r)dM_s\,|\, \mathcal F_r]=0$. Hence developing squares and using the tower property we have
\begin{align*} \mathbb E\left[\left|\sum_{\ell=1}^n(M_{\frac{\ell}{n}}- M_{\frac{(\ell-1)}{n}})^2 -\int_0^1\sigma(M_t)^2dt\right|^2\right] =&\sum_{\ell=1}^n \mathbb E\left[\left|(M_{\frac{\ell}{n}}-M_{\frac{(\ell-1)}{n}})^2-\int_{\frac{(\ell-1)}{n}}^{\frac{\ell}{n}}\sigma(M_s)^2ds\right|^2\right]\leq \frac{c}{n}.
\end{align*}
A Borel-Cantelli argument yields that $\sum_{\ell=1}^{2^k} |M_{(\ell-1)\cdot 2^{-k}}-M_{\ell\cdot 2^{-k}}|^2\to \int_0^1\sigma(M_t)^2dt$ almost-surely, since the sequence $\{2^{-k}\}_k$ is summable.}

All in all, we have proved the almost-sure bound
\[\limsup_k\frac{1}{2^k}\log\left( \frac{ d \mathbb Q}{d\mathbb P}\Big \vert_{\mathcal F^{2^k}} \Big . \big (M_{\ell\cdot 2^{-k}}\big )_{\ell=0}^{2^k}  \right)\leq \frac{1}{2}\int_0^1 \left\{\sigma(M_s)^2 -1 - \log\sigma(M_s)^2 \right\}\,ds .\]
Since similar arguments deal with the lower bound and the limes inferior, we conclude.
\end{proof}

\bibliography{name}
\bibliographystyle{plain}

\end{document}